\documentclass[11pt,reqno]{amsart}
\usepackage[breaklinks]{hyperref}
\usepackage{amsaddr}
\usepackage{mathrsfs}
\usepackage{url}
\usepackage{comment}
\usepackage{color}

\usepackage{url}
\usepackage{amssymb}
\renewcommand{\Re}{\operatorname{Re}}

\numberwithin{equation}{section}
\theoremstyle{plain}
\newtheorem{theorem}{Theorem}[section]
\newtheorem{lemma}[theorem]{Lemma}
\newtheorem{corollary}[theorem]{Corollary}

\newcommand{\Q}{\mathbb{Q}}

\newcommand{\ii}{\mathbf{i}}
\newcommand{\JJ}{\mathbf{j}}
\newcommand{\ZZ}{\mathbf{z}}
\newcommand{\AAA}{\mathbf{A}}
\newcommand{\trans}{^{\!\top\!}}
\newcommand{\kron}[2]{\delta(#1,#2)}
\newcommand{\Bpol}[1]{\mathbf{B}_{#1}}
\newcommand{\hsum}[1]{\sum_{\ell=1}^{\infty}\frac{(-1)^{\ell}H_{\ell}}{\ell^{#1}}} 

   \makeatletter
\def\proof{\@ifnextchar[{\@oproof}{\@nproof}}
\def\@oproof[#1][#2]{\trivlist\item[\hskip\labelsep\textit{Proof of\
#1.}~]\ignorespaces}
\def\@nproof{\trivlist\item[\hskip\labelsep\textit{Proof.}~]\ignorespaces}

\makeatother

\usepackage[dvipsnames,svgnames,x11names,hyperref]{xcolor}
\hypersetup{breaklinks,
  colorlinks   = true, 
  urlcolor     = blue, 
  linkcolor    = blue, 
  citecolor    = RoyalPurple 
}

\begin{document}
\title[A family of integrals related to values of the Riemann zeta function]{A family of integrals related to values of the Riemann zeta function}

\author{Rahul Kumar}\thanks{2010 \textit{Mathematics Subject Classification.}
 Primary 11M06, 26B15; Secondary 33B15.\\
\textit{Keywords and phrases.} Riemann zeta function, integral evaluations, alternating harmonic number, Euler sums, Polylogarithmic integrals}
\address{Department of Mathematics, The Pennsylvania State University, University Park, PA, U.S.A.}
\email{rjk6031@psu.edu}

\author{Paul Levrie}
\address{Faculty of Applied Engineering, University of Antwerp, Groenenborgerlaan 171, 2020 Antwerpen, Belgium}
\address{Department of Computer Science, KU Leuven, P.O. Box 2402, 3001 Heverlee, Belgium}
\email{paul.levrie@kuleuven.be}
\author{Jean-Christophe Pain}
\address{CEA, DAM, DIF, F-91297 Arpajon, France}
\address{Universit\'e Paris-Saclay, CEA, Laboratoire Mati\`ere en Conditions Extr\^emes, F-91680 Bruy\`eres-le-Ch\^atel, France}
\email{jean-christophe.pain@cea.fr}

\author{Victor Scharaschkin}
\address{Department of Mathematics. University of Queensland. St Lucia, Australia, 4072}
\email{v.scharaschkin@gmail.com}

\begin{abstract}
We propose a relation between values of the Riemann zeta function $\zeta$ and a family of integrals. This results in an integral representation for $\zeta(2p)$, where $p$ is a positive integer, and an expression of $\zeta(2p+1)$ involving one of the above mentioned integrals together with a harmonic-number sum. Simplification of the latter eventually leads to an integral representation of $\zeta(2p+1)$.
\end{abstract}

\maketitle

\section{Introduction}\label{intro}

In this paper we consider integrals of the form 
\begin{align}\label{mainint}
\int_{-\infty}^\infty \frac{t^{n}\log\left(1+e^{t}\right)}{1+e^t}\, dt.
\end{align}
These have been studied previously in statistical plasma physics, particularly in the so-called Sommerfeld temperature-expansion of the electronic entropy, where they play a major role in the calculation of the equation of state of dense plasmas \cite{arnault2023}.  In this regard an integral representation for $\zeta(4)$, $\zeta$ being the Riemann zeta function, was obtained in \cite{pain2023}.  The corresponding expression was obtained using relations for polylogarithms \cite{apostol2010,bailey1997,lewin1981,lewin1991,srivastava2012} and a possible generalization to any even argument of the zeta function was also considered there.  The formula for $\zeta(4)$ reads (\cite{pain2023}, Eq. (32)]):
\[
\zeta(4)=\dfrac17\int_{-\infty}^{+\infty}
\dfrac{t^2\ln (1+e^t)}{1+e^t}dt,
\]
which also follows from (\cite{lewin1981}, Eqs. (7.61), (7.62) and (7.65)]):
\[
\zeta(4)=-\dfrac12\int_{0}^{1}\dfrac{\ln^2 t\ln (1-t)}{t}dt
=-\dfrac16\int_{0}^{1}\dfrac{\ln^3 t}{1-t}dt
=-2\int_{0}^{1}\dfrac{\ln t\ln^2 (1-t)}{t}dt.
\]

Recall Euler's famous result \cite{edwards2001} that the Riemann zeta function for even values of the argument is given by
\begin{equation}\label{eq-even-zeta-Euler}
\zeta(2k)=\sum_{n=1}^{\infty}\frac1{n^{2k}} \ = \ \frac{(-1)^{k+1}B_{2k} \ (2\pi)^{2k}} {2 \, (2k)!}, \qquad k\geq 1 
\end{equation}
where $B_j$ are the Bernoulli numbers.  The case of odd values of the argument is much more complicated, although some expressions exist (see for instance \cite{chamberland2011}).  A number of integral representations of the Riemann zeta function also exist. For instance, in the case where $\Re(s) > 1$, the following representation has been known since Euler:
\begin{equation}
\zeta(s)={\frac {1}{\Gamma (s)}}\int _{0}^{1}{\frac {(-\log u)^{s-1}}{1-u}}\,du,
\end{equation}
where $\Gamma$ represents the usual Gamma function.  Making a change of variables, the latter can be written as
\begin{equation}\label{eq-zeta-Gamma}
\zeta(s)=\frac{1}{\Gamma(s)}\int_0^{\infty}\frac{t^{s-1}}{e^t-1}\,dt.
\end{equation}
Equation~(\ref{eq-zeta-Gamma}) can be interpreted as a particular case of a general transformation of Dirichlet series.  Indeed, $s \longmapsto \zeta(s)\Gamma(s)$ is the Mellin transform of the function $t\mapsto 1/(e^t-1)$.  Alternatively, using the Euler-Maclaurin formula \cite{cartier1992, tenenbaum2008}, one obtains, for $\Re(s)>1$:
\begin{equation*}
\zeta(s) = \frac{1}{s-1}+ \frac{1}{2} +\sum _{k=2}^{n}\textstyle \displaystyle\frac{(s)_{k-1}}{k! \vphantom{\int^1}} \,B_{k}
-{\displaystyle\frac {(s)_{n}}{n! \vphantom{\int^1}}}\displaystyle \!\int _{1}^{\infty }\frac{\Bpol{n}\bigl(x-\lfloor x \rfloor\bigr)}{x^{n+s}}\,dx,
\end{equation*}
where $\lfloor x\rfloor$ denotes the integer part of $x$, $(a)_n=a(a+1)\cdots (a+n-1)$ the Pochhammer symbol and the $\Bpol{n}(x)$ are the usual Bernoulli polynomials
\begin{equation}\label{eq-BernPolyDef}
\Bpol{n}(x) = \sum\limits_{k=0}^n \binom{n}{k}B_{n-k}\, x^k 
\end{equation}
(with the convention that $B_1=-1/2$).

There are many other intriguing integrals for $\zeta$ values in the literature.  Borwein and Borwein obtained the following \cite{borwein1995}:
\begin{equation}
\zeta(4)=\frac{2}{11\pi}\int_0^{\pi}\theta^2\log^2\left[2\cos\left(\frac{\theta}{2}\right)\right]d\theta,
\end{equation}
while in \cite[Eq. (1.16)]{sofo2019} Sofo gives:%
\begin{equation}
\zeta(4) = -\frac{8}{3}\int_0^1 \frac{\log(1-x)(\log(1+x))^2}{x}\,dx. 
\end{equation}
In this work, we present the following representations of the Riemann zeta function:
\begin{theorem}\label{even}
For any natural number $n$, we have
\begin{align}\label{even eqn}
\zeta(2n)=\frac{1}{n(2-2^{2-2n})\Gamma(2n-1)}\int_{-\infty}^\infty \frac{t^{2n-2}\log\left(1+e^{t}\right)}{1+e^t}dt.
\end{align}
\end{theorem}
The counter-part of the above result for odd arguments is more involved (as usual).  To state it, we need the alternating Euler sum:
\begin{align}\label{halphas}
\hsum{s}\qquad (\mathrm{Re}(s)>1),
\end{align}
where $H_{\ell}$ denotes the $\ell$-th harmonic number, given by
\begin{align*}
H_{\ell}:=\sum_{j=1}^{\ell}\frac{1}{j}
\end{align*}
(with $H_0$ defined to be~$0$).

We shall show the following.

\begin{theorem}\label{odd}

For any natural number $n$, we have
\begin{align*}
\zeta(2n+1)=&-\frac{2^{2n+1}}{(2n+1)(2^{2n}-1)}\hsum{2n}\\
&+\frac{2^{2n}}{(2n+1)\Gamma(2n)(2^{2n}-1)}\int_{-\infty}^\infty \frac{t^{2n-1}\log\left(1+e^{t}\right)}{1+e^t}dt.
\end{align*}
\end{theorem}

In fact, we derive the following more general result which gives Theorems~\ref{even} and Theorem~\ref{odd} as special cases.

\begin{theorem}\label{general}
Let $n$ be any natural number greater than $1$.  Then we have
\begin{align}\label{general eqn}
\zeta(n)=\frac{\bigl((-1)^n-1\bigr)}{n(1-2^{1-n})}\hsum{n-1} +\frac{1}{n\Gamma(n-1)(1-2^{1-n})} \int_{-\infty}^\infty \frac{t^{n-2}\log\left(1+e^{t}\right)}{1+e^t}dt.
\end{align}
\end{theorem}

Harmonic sums, such as the one occurring here, have been intensively studied.  The original Euler formula reads (cf. \cite[p. 252]{berndt1985}):
\[
2\sum_{\ell=1}^{\infty}\dfrac{H_{\ell}}{\ell^{m}}=(m+2)\zeta(m+1) - \sum_{k=1}^{m-2}\zeta(m-k)\zeta(k+1),\qquad m\geq 2.
\]
The literature on the subject post-Euler is abundant; one can mention the work of Srivastava on zeta and $q$-zeta functions and associated series and integrals \cite{srivastava2001,srivastava2012}, the contribution of Batir to combinatorial identities and harmonic sums \cite{batir2017}, and the difficult integrals, sums and series collected by V\v{a}lean \cite{valean2019}.  As concerns the closely related topic of polylogarithmic integrals, Sofo obtained some Bailey-Borwein-Plouffe-type series \cite{sofo2021b}, and Mez\H o studied a family of polylog-trigonometric integrals \cite{mezo2018}.

In our case we require the sum $\hsum{n-1}$, but as this occurs only when $n$ is odd, we actually only need $\hsum{2m}$.   This sum has been determined previously, in 1987 by Sitaramachandrarao~\cite[3.17]{sitaramachandrarao1987}, who denoted it $-H_7(2m)$, and in 1998 by Flajolet and Salvy~\cite[Theorem 7.1(ii)]{flajolet1998}), who denoted it $-S^{+-}_{1,2m}$ (Warning: both references refer to a function denoted $\overline{\zeta}$, whose meaning is different in each paper).  Their result can be written in the form
\begin{equation}\label{eq-H-yuck}
\hsum{2m} = \biggl[m -\frac{2m+1}{2^{2m+1}}\biggr] \zeta(2m+1) -\sum_{k=1}^{m-1} (1-2^{1-2k} )\zeta(2k)\zeta(2m-2k+1),
\end{equation}
which can be used in~(\ref{general eqn}) to express the integral purely in terms of zeta functions; see~(\ref{eq-odd-initial}) below for details.

The proofs of Theorems \ref{even}, \ref{odd} and \ref{general} are provided in section \ref{proofs}.  Observe that the integrals occurring above are the moments of $\log(1+e^x)/(1+e^x)$.  By taking appropriate linear combinations of these, we can obtain an integral representation of $\zeta(2n+1)$. Curiously the required integrand turns out to be the product of $\log(1+e^x)/(1+e^x)$ with a polynomial related to the tangent function.

\begin{theorem}\label{withQn}
Let $n$ be a natural number. Then we have
\begin{equation*}
\zeta{(2n+1)}= \frac{1}{(2n-1)!}\,\pi^{2n-1}\!\! \int_{-\infty}^{\,\infty} \!\! Q_n\!\left(\frac{x}{\pi}\right) \frac{\log\bigl(1+e^x\bigr)}{1+e^x}\, dx.
\end{equation*}
where the monic polynomial $Q_m\in \Q[x]$ of degree $2m-1$ is defined as
\begin{equation}\label{eq-Q}
Q_m(t) = \frac{(-1)^m}{2m}\sum\limits_{k=1}^m (-1)^{k}\binom{2m}{2k-1}t^{2k-1}.
\end{equation}
The first few $Q_m$ are $Q_1(t)=t$, $Q_2(t)=t^3-t$, $Q_3(t) = \frac{1}{3}(3t^5-10t^3+3t)$.
\end{theorem}

An explicit calculation (see Lemma~\ref{Lem-Qnroots}) reveals that the zeros of $Q_m$ are all of the form $\tan(k\pi/2m)$.  We thus obtain the next corollary.
\begin{corollary}\label{cor_zeta_odd}
Let $n$ be a natural number. At odd arguments we have the following integral representation of zeta:
\begin{equation}\label{eq-zeta-odd}
\zeta(2n+1) = \frac{1}{(2n-1)!} \int_{\!-\infty}^{\,\infty} \prod_{k=1-n}^{n-1} \biggl[x-\pi\tan\Bigl(\frac{k\pi}{2n}\Bigr)\biggr]\; \frac{\log(1+e^x)}{1+e^x}\,dx.
\end{equation}
\end{corollary}
\vspace{5mm}
The proofs of Theorem \ref{withQn} and Corollary \ref{cor_zeta_odd} are given in the next section as well.

\section{Proofs}\label{proofs}

\begin{proof}[Theorem \ref{general}][]

We observe that
\begin{align}
I(n) := \int_{-\infty}^\infty \frac{t^{n}\log\left(1+e^{t}\right)}{1+e^t}dt&=I_1(n)+(-1)^nI_2(n),
\end{align}
where
\begin{equation}
I_1(n):=\int_{0}^\infty \frac{t^{n}\log\left(1+e^{t}\right)}{1+e^t}dt,\qquad
I_2(n):=\int_{0}^\infty \frac{t^{n}\log\left(1+e^{-t}\right)}{1+e^{-t}}dt.
\end{equation}
By writing $1+e^t$ as $e^t(1+e^{-t})$, the integral $I_1(n)$ can be written as $I_1(n)=I_3(n)+I_4(n)$, where
\begin{equation}
I_3(n) := \int_{0}^\infty \frac{t^{n}\log\left(1+e^{-t}\right)}{1+e^{t}}dt,\qquad I_4(n) :=\int_{0}^\infty \frac{t^{n+1}}{1+e^{t}}dt.
\end{equation}
Thus
\begin{equation}
I(n) = \Bigl(I_2(n) + I_3(n)\Bigr) + I_4(n) +\Bigl((-1)^n-1\Bigr)I_2(n).
\end{equation}
We now perform integration by parts on $I_2(n)+I_3(n)$. Since $\frac{1}{1+x \vphantom{\int}} + \frac{1}{1+x^{-1} \vphantom{\int}} = 1$,
\begin{eqnarray*}
I_2(n)+I_3(n) &=& \int_0^{\infty} t^n\log(1+e^{-t})\, dt \\
&=& %
\frac{1}{n+1} \int_0^{\infty} \log(1+e^{-t})\, d\, (t^{n+1}) \\ &=& %
\frac{1}{n+1}\Bigl[t^{n+1}\log(1+e^{-t})\Bigr]_{0}^{\infty} \quad -\quad \frac{1}{n+1} \int_0^{\infty} \frac{(-e^{-t})}{1+e^{-t}}\, t^{n+1}\, dt \\ &=& %
\frac{1}{n+1}\lim\limits_{t\to \infty} \Bigl[t^{n+1}\log(1+e^{-t})\Bigr]\quad +\quad\frac{1}{n+1} \int_0^{\infty} \frac{t^{n+1}}{1+e^{t}}\, dt \\ &=& %
\lim\limits_{t\to \infty} \Bigl[t^{n+1}\log(1+e^{-t})\Bigr]\quad +\quad\frac{1}{n+1} I_4(n).
\end{eqnarray*}
Now $\log(1+x)<x$ so for $t>0$ we have $0<t^{n+1}\log(1+e^{-t}) < e^{-t}\,t^{n+1} \to 0$ as $t \to \infty$. Thus
\begin{equation}\label{eq-InSimplified}
I(n) = \frac{n+2}{n+1}I_4(n) +\Bigl((-1)^n-1\Bigr)I_2(n).
\end{equation}

We obtain $I_4(n)$ immediately upon recalling the Dirichlet $\eta$ function (alternating $\zeta$ function) defined by \cite{spiegel1968}:
\begin{equation}
\eta(s) = \frac{1}{\Gamma(s)}\int_0^{\infty} \frac{x^{s-1}}{1+e^x}\,dx
\end{equation}
and satisfying
\begin{equation}\label{eq-eta}
\eta(s) =\sum_{\ell=1}^{\infty} \frac{(-1)^{\ell+1}}{\ell^s} = (1-2^{1-s})\zeta(s).
\end{equation}
Interestingly, the Dirichlet $\eta$ function also appears in mathematical physics as the integral of a Fermi--Dirac distribution (see \cite[7.5]{dougall1938}).  

We obtain
\begin{equation}\label{eq-i4}
I_4(n) = \Gamma(n+2)\eta(n+2) = (1-2^{-1-n})\zeta(n+2)\Gamma(n+2).
\end{equation}
Making the change of variable $u=e^{-t}$ in $I_2(n)$, we get
\begin{equation}
I_2(n) =(-1)^n\int_0^1 \frac{ \log^n(u)\log(1+u)}{u(1+u)}\, du.
\end{equation}
We now obtain $I_2(n)$ from the following, which is a special case of a result due to Sofo and Batir \cite[p.~409, Theorem 3.5]{sofo2021}. Namely, for $b\geq-2$, and $n \geq 0$ an integer,
\begin{align}\label{anba2}
\int_0^1\frac{x^b\log^n(x)\log(1+x)}{1+x}dx=(-1)^n n!\sum_{\ell=1}^\infty\frac{(-1)^{\ell+1}H_\ell}{(\ell+1+b)^{n+1}}.
\end{align}
Putting $b=-1$ in~(\ref{anba2}) gives
\begin{equation}\label{eq-i2}
I_2(n) = (-1)^n\!\int_0^1\frac{\log^n(x)\log(1+x)}{x(1+x)}\, dx = -\Gamma(n+1)\hsum{n+1}.
\end{equation}
We note in passing that~(\ref{anba2}) substituting $u=e^{-t}$ allows us to find $I_3(n)$ as well (although we do not need it).  Namely, putting $b=0$ in~(\ref{anba2}) and simplifying yields
$$
I_3(n) = \Gamma(n+1)\hsum{n+1} +\Gamma(n+1)(1-2^{-1-n})\zeta(n+2).
$$
Finally equations~(\ref{eq-InSimplified}), (\ref{eq-i4}) and~(\ref{eq-i2}) give
\begin{equation}
I(n) = (n+2)(1-2^{-1-n})\Gamma(n+1)\zeta(n+2) -\bigl((-1)^n-1\bigr)\Gamma(n+1)\hsum{n+1}.
\end{equation}
Replacing $n$ by $n-2$ and solving for $\zeta(n)$ gives the result of Theorem~\ref{general}.
\end{proof}
\vspace{5mm}
\begin{proof}[Theorem \ref{withQn}][]

Before commencing the calculation, we collect some facts about Bernoulli numbers.  Recall that
\begin{equation}
\sum_{s=0}^n \binom{n+1}{s}B_s = \kron{n}{0} \label{eq-BernRecur}
\end{equation}
where $\kron{p}{q}$ denotes Kronecker delta, and
\begin{equation}
\Bpol{n}(1-x)=(-1)^n\,\Bpol{n}(x) \label{eq-BernoulliPoly}
\end{equation}
where $\Bpol{n}$ denotes the $n$th Bernoulli polynomial, defined in~(\ref{eq-BernPolyDef}). We then have:
\begin{lemma}\label{Lem-BernoulliNumbers}\ 
\begin{enumerate}
\item $\displaystyle2\sum\limits_{i=0}^m\binom{2m+1}{2i}B_{2i}=(2m+1) +\kron{m}{0}.$
\item $\displaystyle\sum\limits_{i=0}^{m} (2^{2i}-2)\binom{2m+1}{2i}B_{2i} =-\kron{m}{0}$.
\end{enumerate}
\end{lemma}
\begin{proof} (a) For $n>0$ the basic recursion~(\ref{eq-BernRecur}) gives
\begin{equation*}
0 \;\;=\;\;\textstyle \sum\limits_{j=0}^n \binom{n+1}{j}B_j \;\;=\;\; %
\textstyle-\frac{1}{2}\binom{n+1}{1} + \sum\limits_{i=0}^{\lfloor n/2\rfloor} \binom{n+1}{2i}B_{2i}
\end{equation*}
\begin{equation*}
\textstyle\therefore\quad \sum\limits_{i=0}^{\lfloor n/2\rfloor} \binom{n+1}{2i}B_{2i} = \textstyle \frac{n+1}{2} \quad
\therefore\quad \textstyle 2\sum\limits_{i=0}^m\binom{2m+1}{2i}B_{2i} =2m+1\qquad m>0.
\end{equation*}
Allowing for the $m=0$ case, this gives~(a).

(b) Suppose $n$ is odd, say $n=2m+1$. Then~(\ref{eq-BernoulliPoly}) implies $B_n(\frac{1}{2})=0$, and hence
\begin{equation*}
\renewcommand{\arraystretch}{2.2}\begin{array}{rclclcl}
\textstyle 0 &=& \textstyle 2^n\,\Bpol{n}(\frac{1}{2}) %
&=& \textstyle \sum\limits_{k=0}^n \binom{n}{k}B_{n-k}2^{n-k} %
&=& \textstyle \sum\limits_{k=0}^n \binom{n}{n-k}B_{n-k}2^{n-k} \\ %
&=& \textstyle \sum\limits_{j=0}^n \binom{n}{j}2^{j}B_{j} %
&=&\textstyle \sum\limits_{j=0}^{2m+1} \binom{2m+1}{j}2^{j}B_{j} %
&=&\textstyle \sum\limits_{i=0}^{m} \binom{2m+1}{2i}2^{2i}B_{2i} + 2\binom{2m+1}{1}B_{1} \\ %
\therefore\quad 2m+1 &=&\textstyle \sum\limits_{i=0}^{m} \binom{2m+1}{2i}2^{2i}B_{2i}.\hspace*{-40mm}
\end{array}
\end{equation*}
Now (b) follows by subtracting (a) from the last equation.
\end{proof}

\vspace{0.5cm}

We now obtain our expression for $\zeta$ at odd arguments. First we substitute the explicit form of $\hsum{n}$ from equation~(\ref{eq-H-yuck}) into Theorem~\ref{odd} to obtain 
\begin{equation}\label{eq-odd-initial}
I(n)= n!\,\sum\limits_{\ell =0}^{\frac{n-1}{2}} \bigl(2-2^{2-2\ell}\bigr)\, \zeta(n+2-2\ell)\, \zeta(2\ell), \qquad n\text{\ odd}
\end{equation}
where we absorbed the $\zeta(n+2)$ term into the sum using the fact that $\zeta(0)=-\frac{1}{2}$. Now we simply view the zeta terms at odd arguments as unknowns, and solve for them using linear algebra. To make this clearer, let $n=2j-1$ and $k=j-\ell$. Then~(\ref{eq-odd-initial}) is effectively a system of linear equations
\begin{equation}\label{eq-odd}
\sum\limits_{k=1}^{j} a_{j,k}\,\zeta(2k+1) =I(2j-1)\quad \text{where}\quad a_{j,k}= (2j-1)!\, \bigl(2-2^{2-2j+2k}\bigr)\, \zeta(2j-2k).
\end{equation}

Let $\AAA_n$ be the $n\times n$ lower triangular matrix with non-zero entries $a_{j,k}$ for $1 \leq k \leq j \leq n$ defined by~(\ref{eq-odd}): %
\begin{equation*}
\AAA_n = \left(\begin{array}{ccccc}
1 & 0 & 0 & 0 & \cdots \\ %
6\,\zeta(2) & 3! & 0 & 0 & \cdots \\ %
210\,\zeta(4) & 120\,\zeta(2) & \quad 5!\quad & 0 & \cdots \\
\vdots & \vdots & \vdots & 7!
& \ddots %
\end{array}\right).
\end{equation*}
Form column vectors $\JJ_n=\bigl(I(1), I(3), \ldots, I(2n-1)\bigr)\trans$ and $\ZZ_n=\bigl(\zeta(3), \zeta(5), \ldots, \zeta(2n+1)\bigr)\trans$. From~(\ref{eq-odd}) we have $\JJ_n = \AAA_n\, \ZZ_n$, so
\begin{equation} \label{eq-zeta-inverse-matrix}
\ZZ_n = \AAA_n^{-1}\,\JJ_n.
\end{equation}
Consider the lower triangular matrix $\AAA_n^{-1}$. We shall show that the non-zero entries\footnote{Because each lower triangular $\AAA_n$ embeds in the next, the entries of each $\AAA_n^{-1}$ are the same for all $n$ for which they are defined.} are
\begin{equation*}
b_{j,k} = \frac{(-1)^{j+k}\, \pi^{2(j-k)}}{\vphantom{\int} (2j-2k+1)!\; (2k-1)!}\qquad \qquad (k\leq j).
\end{equation*}
To establish this claim it suffices to show\footnote{Recall $CD=I \implies DC=I$ automatically for square matrices.} $\sum\limits_{\ell=k}^j a_{j,\ell}\, b_{\ell,k} =\kron{j}{k}$ for each $(j,k)$ with $1\leq k\leq j \leq n$. Using~(\ref{eq-even-zeta-Euler}) we have
\begin{equation*}
\renewcommand{\arraystretch}{3.4}\begin{array}{rcl} %
\displaystyle \sum\limits_{\ell=k}^j a_{j,\ell}\, b_{\ell,k} \!\!\!\!\! &\!\!=\!\!&\!\!\! \displaystyle %
\sum\limits_{\ell=k}^j (2j-1)!(2-2^{2-2j+2\ell})\zeta\bigl(2(j-\ell)\bigr) \;\;\cdot\;\; \frac{(-1)^{\ell+k} \pi^{2(\ell-k)}}{\vphantom{\int} (2\ell-2k+1)!\; (2k-1)!} \\ & \stackrel{(\ref{eq-even-zeta-Euler})}{\!\!=\!\!} &\!\!\! \displaystyle %
\sum\limits_{\ell=k}^j \frac{ (2j-1)!\, (2-2^{2-2j+2\ell})\; %
(-1)^{j-\ell+1}B_{2(j-\ell)}\, 2^{2j-2\ell}\, \pi^{2j-2\ell} \; (-1)^{\ell+k}\, \pi^{2\ell-2k}} {2\, (2j-2\ell)!\; (2\ell-2k+1)!\; (2k-1)!} \\ \!\!\!\!\! &\!\!=\!\!&\!\!\! \displaystyle %
(-1)^{j+k+1} \pi^{2j-2k} \frac{(2j-1)!}{(2k-1)!} \sum\limits_{\ell=k}^j \frac{(2^{2j-2\ell}-2)} {(2j-2\ell)!\; (2\ell-2k+1)!} B_{2(j-\ell)}.
\end{array}
\end{equation*}
Put $m=j-k\geq 0$ and $i=j-\ell$. The right-hand side becomes
\begin{equation*}
(-1)^{m+1}\pi^{2m} \frac{(2j-1)!} {\vphantom{\int} (2k-1)!\, \;(2m+1)!}\sum\limits_{i=0}^{m} (2^{2i} -2)\binom{2m+1}{2i}B_{2i} \;\;\stackrel{(\ref{Lem-BernoulliNumbers})(b)}{=}\;\; \kron{j}{k}
\end{equation*}
as required.

The bottom row of~(\ref{eq-zeta-inverse-matrix}) thus gives 
\begin{equation*}
\renewcommand{\arraystretch}{2.2}\begin{array}{rcl}
\zeta(2n+1) &=& \displaystyle %
\sum_{k=1}^n \frac{(-1)^{n+k} \pi^{2(n-k)}}{\vphantom{\int} (2n-2k+1)!\; (2k-1)!}\,I(2k-1) \\ &=& \displaystyle %
\int_{-\infty}^{\,\infty} \sum\limits_{k=1}^n \frac{(-1)^{n+k}} {\vphantom{\int} (2n-2k+1)!\; (2k-1)!}\; \pi^{2(n-k)} x^{2k-1} \frac{\log(1+e^x)} {1+e^x}\, dx \\ &=& \displaystyle %
\frac{1}{(2n-1)!}\frac{(-1)^n \pi^{2n-1}}{2n} \int_{-\infty}^{\,\infty} \sum\limits_{k=1}^n (-1)^{k}\binom{2n}{2k-1} \left(\frac{x}{\pi}\right)^{\!2k-1} \frac{\log(1+e^x)}{1+e^x}\, dx \\ &=& \displaystyle %
\frac{1}{(2n-1)!}\,\pi^{2n-1} \int_{-\infty}^{\,\infty} \!\! Q_n\!\left(\frac{x}{\pi}\right) \frac{\log(1+e^x)}{1+e^x}\, dx,
\end{array}
\end{equation*}
which completes the proof of Theorem \ref{withQn}.

\end{proof}

\vspace{0.5cm}

\begin{proof}[Corollary \ref{cor_zeta_odd}][]

This result follows immediately once we find the roots of $Q_n$ explicitly.  To that end, first recall
\begin{equation*}
\tan(-\theta) = -\tan \theta = -\frac{\frac{1}{2\ii}(e^{\ii\theta} -e^{-\ii\theta})} {\frac{1}{2}(e^{\ii\theta} +e^{-\ii\theta})} = \ii \frac{e^{\ii\theta} -e^{-\ii\theta}} {e^{\ii\theta} +e^{-\ii\theta}} = -\ii \frac{1 -e^{2\ii\theta}} {1 +e^{2\ii\theta}}.
\end{equation*}
If $\theta = \frac{k\pi}{q}$ and $\omega_q= e^{\frac{2\pi\ii}{q}}$ is a primitive $q$th root of unity then $e^{2\ii\theta} = e^{\frac{2k\pi\ii}{q}} =\omega_q^k$, and hence
\begin{equation}\label{eq-tan}
\tan\Bigl(\frac{-k\pi}{q}\Bigr) = -\ii \frac{1_{\vphantom{\int}} -\omega_q^k } {1 +\omega_q^k \vphantom{\int^2} }.
\end{equation}

\begin{lemma}\label{Lem-Qnroots}
The roots of $Q_n$ are $\tan\bigl(\frac{k\pi}{2n}\bigr)$ for $-(n-1) \leq k \leq n-1$.
\end{lemma}
\begin{proof} Equivalently, we show the roots are $\tan\bigl(-\frac{k\pi}{2n}\bigr)$ for $-(n-1) \leq k \leq n-1$. Consider
\begin{equation*}
\renewcommand{\arraystretch}{2.1}\begin{array}{rcl}
\displaystyle \frac{(-1)^n\,\ii}{4n} \Bigl[(1+\ii t)^{2n}-(1-\ii t)^{2n}\Bigr] &=& \displaystyle %
\frac{(-1)^n\, \ii}{4n} %
\sum_{j=0}^{2n} \binom{2n}{j}\Bigl( \ii^j- (-\ii)^j\Bigr)\,t^j \\ &=& \displaystyle %
\frac{(-1)^n\,\ii}{2n} %
\sum_{k=1}^{n} \binom{2n}{2k-1}\ii(-1)^{k-1}\, t^{2k-1}\\ &=& \displaystyle %
\frac{(-1)^{n}}{2n} %
\sum_{k=1}^{n} \binom{2n}{2k-1}(-1)^{k}\,t^{2k-1}\\[-3mm] &=& \displaystyle %
Q(t).
\end{array}
\end{equation*}
Thus $Q_n(t)=0 \iff (1+\ii t)^{2n}=(1-\ii t)^{2n} \iff \frac{1 -\ii t}{1 +\ii t\vphantom{\int^2}}=\omega_{2n}^k$ for some $k$ with $-(n-1)\leq k \leq n-1$, which holds \textsc{iff} $t= -\ii\frac{1 -{\omega_{2n}^k}_{\vphantom{g}}} {1 +\omega_{2n}^k \vphantom{\int^2}} =\tan\bigl(\frac{-k\pi}{2n}\bigr)$ from~(\ref{eq-tan}).

\end{proof}

\vspace{0.5cm}

Since $Q_n$ is monic of degree $2n-1$ and we have exhibited $2n-1$ distinct roots we obtain expression (\ref{eq-zeta-odd}), which completes the proof of Corollary \ref{cor_zeta_odd}.
\end{proof}

As an example,
\begin{equation*}
\zeta(5) = \frac{1}{6} \!\!\int\limits_{-\infty}^{\infty} (x^3-\pi^2 x)\frac{\log(1+e^x)}{1+e^x}\,dx.
\end{equation*}

\section{Examples, potential consequences and generalizations}\label{persp}

The first values in terms of zeta functions of the integral (\ref{mainint}) are provided in the table below.

\begin{table}[!ht]
\begin{tabular}{|c|c|}\hline
$n$ & $\qquad\displaystyle \int_{-\infty}^{\, \infty \vphantom{\int^3}} \!z^{n}\,\frac{\log(1 +e^z)} {1+e^z\vphantom{\int^1}}\,dz \qquad$ \\[2mm] \hline 
$0$ & $\zeta(2)\vphantom{\int^3}$ \\[2mm] 
$1$ & $\zeta(3)$ \\[2mm] 
$2$ & $7\zeta(4)$ \\[2mm] 
$3$ & $6\bigl[\zeta(5) + \zeta(2)\zeta(3)\bigr]$ \\[2mm]
$4$ & $\displaystyle\frac{279}{2}\zeta(6)$ \\[4mm]
$5$ & $30\bigl[ 4\zeta(7) +4\zeta(2)\zeta(5) +7\zeta(3)\zeta(4)\bigr]$ \\[2mm]
\hline
\end{tabular} 
\end{table}
Following Euler, from~(\ref{eq-even-zeta-Euler}) we know that $\zeta(n)$ is transcendental for all even~$n$. It is conjectured that all values at odd integers are irrational and even algebraically independent over the field $\Q(\pi)$, and thus all transcendental~\cite{fischler2002}, but progress towards these conjectures has been limited. Famously in 1978 Ap\'ery proved that $\zeta(3)$ is irrational \cite{apery1979,poorten1979}. No other explicit odd~$n$ is known for which $\zeta(n)$ has been proved to be irrational, although in 2000, Rivoal proved that there exists an infinity of irrational numbers among odd integer values of the zeta function~\cite{rivoal2000}, and one year later Zudilin showed that at least one of the four numbers $\zeta(5)$, $\zeta(7)$, $\zeta(9)$, $\zeta(11)$ is irrational~\cite{zudilin2001}. These proofs are often formulated by writing $\zeta(n)$ in terms of certain integrals, and obtaining rational approximations thereby. Perhaps integrals related to the ones considered in this paper may be of use in these endeavours.

We can also consider generalizations. In the even argument case, it is fairly easy to extend our argument to show that
\begin{equation}
\int_{-\infty}^{\,\infty}\!\! x^{2n}\frac{\log(1+e^{ax})}{1+e^{bx}}dx=(2n)!\,\left(1-2^{-2n-1}\right)\frac{\bigl[(2n+1)a^{2n+2}+b^{2n+2}\,\bigr]}{a^{2n+1}b^{2n+2}}\zeta(2n+2),
\end{equation}
where $a,b>0$. We could not find any simple generalization in the odd case, only some particular values, such as
\begin{equation}
\int_{-\infty}^{\infty}x\frac{\log(1+e^{x/2})}{1+e^{x}}dx=-\frac{\pi^2}{4}\log 2 - \frac{3}{8}\zeta(3),
\end{equation}
or the following formulas connecting $\pi$, the Ap\'ery constant $\zeta(3)$ and the Catalan constant $G$:
\begin{equation}
\int_{-\infty}^{\, \infty}x\frac{\log(1+e^{2x})}{1+e^{x}}dx=\frac{1}{16}\left[16\pi G+2\pi^2\log 2-3\zeta(3)\right]
\end{equation}
and
\begin{equation}
\int_{-\infty}^{\, \infty}x\frac{\log(1+e^{2x})}{1+e^{3x}}dx=-\frac{1}{144}\left[16\pi G-6\pi^2\log 2+\zeta(3)\right].
\end{equation}
It would be interesting to know for which integers $a$, $b$ and $n$ there is an explicit closed form for 
\begin{equation*}
\int_{-\infty}^{\, \infty}\!\! x^{2n+1} \frac{\log(1+e^{ax})}{1+e^{bx}}\, dx.
\end{equation*}

\section{Conclusion}

A relation between values of the Riemann zeta function $\zeta$ and a family of integrals, encountered in the Sommerfeld expansion of thermodynamic quantities in statistical physics was discussed. It provides an integral representation of $\zeta(2n)$, where $n$ is an integer, and an expression of $\zeta(2n+1)$ involving one of the above mentioned integrals together with a sum over harmonic numbers. The latter expression eventually leads to an integral representation of $\zeta(2n+1)$ as well. Possibly such integral representations may be useful for investigations related to the potential irrationality of odd values of the Riemann zeta function. 

\begin{center}
\textbf{Acknowledgements}
\end{center}

\noindent 
We would like to acknowledge useful discussions with A. Sofo, G. Herv\'e, V. N. P. Anghel, D. McNeil, Kwang-Wu Chen and S. Yurkevich. We are also indebted to the anonymous referee for useful comments and suggestions.

\end{document}